\documentclass[11pt]{amsart}
\usepackage{latexsym}
\usepackage{amssymb}
\usepackage{amsmath}
\usepackage{amsthm}
\usepackage{enumerate}
\usepackage{bbm}
\usepackage{verbatim}
\usepackage{xcolor}
\usepackage{yhmath}

\newtheorem{quest}{Question}
\newtheorem{thm}{Theorem}[section]

\newtheorem{cor}[thm]{Corollary}
\newtheorem{prop}[thm]{Proposition}

\theoremstyle{definition}

\newtheorem{defn}[thm]{Definition}

\theoremstyle{remark}
\newtheorem{rem}[thm]{Remark}

\numberwithin{equation}{section}

\def\1{\mathbbm{1}}
\renewcommand{\qed}{\unskip\nobreak\quad\qedsymbol}

\newcommand{\R}{{\mathbb{R}}}
\newcommand{\T}{{\mathbb{T}}}

\newcommand{\Z}{{\mathbb{Z}}}
\newcommand{\N}{{\mathbb{N}}}

\newcommand{\de}{{\mbox{ d}}}
\allowdisplaybreaks

\title{Pointwise Convergence of Sequences of Singular Measures}

\author{Andrew Parrish}
\address{Department of Mathematics and Computer Science\\
Eastern Illinois State University \\
600 Lincoln Avenue\\
Charleston, IL 61920-3099}
\email{ajparrish@eiu.edu}

\author{Joseph Rosenblatt}
\address{Department of Mathematics\\
University of Illinois at Urbana-Champaign\\
1409 W. Green Street\\
Urbana, IL 61801}
\email{rosnbltt@illinois.edu}

\copyrightinfo{2018}{American Mathematical Society}

\date{\today}

\begin{document}
\maketitle

\begin{abstract} We investigate the almost everywhere convergence of sequences of convolution operators given by probability measures $\mu_n$ on $\mathbb R$.  If this sequence of operators constitutes an approximate identity on  a particular class of functions $\mathcal F$, under what additional conditions do we have $\mu_n\ast f \to f$ a.e. for all $f\in \mathcal F$?  We focus on the particular case of a sequence of contractions $C_{t_n}\mu$ of a single probability measure $\mu$, with $t_n\to 0$, so that that the sequence of operators is an approximate identity .  
\end{abstract}

\section{Introduction} \label{intro}
Suppose $\mu_n$ is a sequence of probability measures on $\mathbb R$.  Assume the sequence is an {\em approximate identity} on $L^1(\mathbb R,\lambda)$ where $\lambda$ is the usual Lebesgue measure on $\mathbb R$.  
This means $\| \mu_n\ast f -f\|_1 \to 0$ as $n\to \infty$ for all $f\in L^1(\mathbb R)$.   Then the question is, under what conditions can one expect the limit
$\lim\limits_{n\to \infty} \mu_n\ast f$ to exist almost everywhere?
These conditions can relate to the behavior of the function (e.g. whether it lies in a particular function space) and/or the nature of the measures $\mu_n$.

For example, the classical Lebesgue differentiation theorem says that if the measures are $\mu_n = \frac 1{\epsilon_n}1_{[0,\epsilon_n]} \lambda$, with $\epsilon_n\to 0$ as $n\to \infty$, then we have a.e. convergence for all $f\in L^1(\mathbb R)$.  However, the situation is dramatically different if $(\mu_n)$ is a sequence of discrete measures with supports converging to $0$.  Then a.e. convergence fails for the generic function $f$.  Such a result was shown for point mass measures $\mu_n$ in Bellow~\cite{Bellow}.  Then Bourgain ~\cite{Bourgain} proved a negative result like this for any averages of such a sequence, for example for  $\mu = \frac 1n \sum\limits_{k=1}^n \delta_{x_n}$ when $x_n\to 0$ as $n\to \infty$.  This negative case was extended by Karagulyan ~\cite{Karagulyan} to any sequence of discrete measures with supports shrinking to the origin.   

In contrast to this, in a type of extension of the Lebesgue differentiation theorem, if the measures are absolutely continuous with respect to Lebesgue measure, then there is always a subsequence for which one gets a.e. convergence.  That best result in this direction is in Kostyukovsky and Olevskii~\cite{KO}.  See also Rosenblatt ~\cite{Rosenblatt}. 

Yet even in the discrete case, interesting questions remain when one considers the conditions (particularly on the behavior of the functions) when one can expect convergence. For example, see Parrish and Rosenblatt~\cite{PR}, where the authors consider the main types of question that can be asked when the operators are given by a sequence of measures that are single point masses.

In this article, the focus shifts to when the measures are singular to Lebesgue measure and are {\em continuous}, that is, have no point masses.  We state some results for general sequences, but we mostly we focus on the case when $\mu_n =C_{t_n}\mu$ are contractions of a fixed measure $\mu$ that give an approximate identity.  

First, as in Kostyukovsky and Olevskii \cite{KO}, if $\mu_n$ are {\em Rajchman probability measures}, we observe that there is always a subsequence $\mu_{n_k}$ such that at least $\mu_{n_k}\ast f\to f$ a.e. with respect to Lebesgue measure for all $f\in L^2(\mathbb R)$.  

But in addition, using Riesz products, we can show that there are continuous, singular probability measures $\mu$ that are not Rajchman measures, and some sequence $t_n\to 0$ as $n\to \infty$, such that $C_{t_n}\mu\ast f \to f$ a.e. for all $f\in L^2(\mathbb R)$.  This answers a question raised in Kostyukovsky and Olevskii~\cite{KO}.   

However, we can also show that there are continuous, singular measures and a particular $t_n\to 0$, so that for any subsequence $t_{n_m}$, there are functions in $f\in L^2(\mathbb R)$, indeed even characteristic functions, such that almost everywhere $C_{t_{n_m}}\mu\ast f$ fails to converge.  This type of example was suggested by a method used in Kostyukovsky and Olevskii~\cite{KO}.

\begin{rem}  We are considering the convolution $\nu \ast f$ for a regular finite Borel measure $\nu$ on $\mathbb R$ and a function $f\in L^p(\mathbb R,\beta,\lambda)$ where $\lambda$ is the Lebesgue measure on the Lebesgue measurable sets $\beta$.  Because our functions $f$ are actually equivalence classes, and because a representative of this class might not be a Borel function, there is some issue with what is meant by the convolution $\nu \ast f$.  This is a traditional topic in harmonic analysis that is addressed by a number of authors.  For completeness, we have included here in Appendix~\ref{details} a discussion of the definition of this convolution with appropriate references.  
\qed
\end{rem}

\section{Rajchman Measures and Good Approximate Identities}\label{rajc}

We generally approach the question of pointwise convergence of a sequence of convolutions via comparison against a sequence that is known to converge, specifically a sequence of Lebesgue derivatives. The difference between the two will be assessed using a square function; this, in turn, will be dealt with via Fourier analysis. 

To illustrate the methods, we first consider a relatively well-behaved family of measures.

\begin{defn} A {\em Rajchman measure} is a Borel probability measure $\mu$ such that $\widehat {\mu}(n) \to 0$ as $|n|\to \infty$.
\end{defn}

Rajchman measures have an interesting history: an excellent introduction is provided by the paper \cite{LyonsSurvey}.  The characterization of these measures was resolved by Lyons in the articles \cite{Lyons1} and \cite{Lyons2}.  While this characterization is very interesting, especially from a descriptive set theory viewpoint, it does not appear to be useful for our purposes.

By the Riemann-Lebesgue Lemma, any Borel measure that is absolutely continuous with respect to Lebesgue measure $\lambda$ is a Rajchman measure.  But it is a classical fact that there are also Rajchman measures that are singular with respect to Lebesgue measure (see \cite{Menshov} for the earliest example).

Given a sequence $(\epsilon_n)$ with $\epsilon_n> 0 $ and $\lim \epsilon_n =0$, we may consider the Lebesgue differentiation operator corresponding to this sequence as a convolution against the measures $\lambda_n$, where 
\begin{equation*}
	\lambda_n(A) = \frac{1}{\epsilon_n} \int_A \1_{[0, \epsilon_n]} \de \lambda. 
\end{equation*}

More explicitly, we have
\begin{equation*}
	D_{\epsilon_n}f(x) = \lambda_n \ast f(x) = \frac{1}{\epsilon_n} \int_\T f(x-t) \1_{[0, \epsilon_n]}(t) \de \lambda(t).
\end{equation*}

Drawing on the example provided by the differentiation operator, we say that a sequence of probability measures $(\mu_n)$ on $[0,1]$ has \emph{supports shrinking to zero} if there is a sequence $\epsilon_n >0$ tending to zero such that 
\begin{equation*}
	\mbox{supp}(\mu_n) \subseteq [0, \epsilon_n].
\end{equation*}
It is easy to show that if $(\mu_n)$ has supports shrinking to zero, then it is an approximate identity for $L^1(\mathbb R)$.

\begin{prop}\label{squares}
	Let $(\mu_n)$ be a sequence of Rajchman measures with supports shrinking to zero. Then there exists a sequence $(\epsilon_k)$ and a subsequence $(n_k)$ such that the square function
	\begin{equation*}
		Sf = \left\{ \sum_{k \in \N} | \lambda_k \ast f(x) - \mu_{n_k} \ast f(x) |^2 \right\}^\frac12 
	\end{equation*}
	has a strong $L^2$ bound:
	\begin{equation*}
		\|Sf\|_2 \le C\|f\|_2  \mbox{ for all } f\in L^2[0,1].
	\end{equation*}
\end{prop}
\begin{proof}
	Leaving aside for the moment the selection of $(\epsilon_k)$ and $(n_k)$, we apply the Plancherel theorem:
	\begin{align*}
		\| Sf(x)  \|_{L^2}^2 &= \int_{\T} \sum_{k \in \N} | (\lambda_k - \mu_{n_k}) \ast f(x) |^2  d\lambda\\
		&\leq  \sum_{k \in \N}  \int_{\T} | (\lambda_k - \mu_{n_k}) \ast f(x) |^2  d\lambda \\
		&= \sum_{k \in \N}  \sum_{j \in \Z} | \widehat {\left (\lambda_k -  \mu_{n_k} \ast f\right )}(j) |^2 \\
		&= \sum_{k \in \N}  \sum_{j \in \Z} | \widehat{(\lambda_k - \mu_{n_k})}(j) |^2 |\widehat{f}(j)|^2\\
		&\leq   \sum_{j \in \Z} |\widehat{f}(j)|^2 \sum_{k \in \N} | \widehat{(\lambda_k - \mu_{n_k})}(j) |^2\\
		&=\sum_{j \in \Z} |\widehat{f}(j)|^2 \sum_{k \in \N} | \widehat{\lambda_k}(j) - \widehat{\mu_{n_k}}(j) |^2. 
	\end{align*}

The task is now to show that we can arrange for the the sum over $k$ to be bounded uniformly in $j$.   First, for any $\lambda_1$ and $\mu_{n_1}$, $|\widehat {\lambda_1}(s) - \widehat {\mu_{n_1}}(s)|\le 4$
for all $s\in \mathbb R$.  Now take as an induction hypothesis: for $K\ge 1$, we have taken a choice of $\lambda_k$ and $\mu_{n_k}$ such that there is a constant $C_K < \sqrt 5$ for which for all $s\in \mathbb R$
\[\Sigma_K(s) = \sum\limits_{k=1}^K |\widehat {\lambda_k}(s) - \widehat {\mu_{n_k}}(s)|^2 \le C_K^2.\] 
Because the measures are all Rajchman measures, if $\epsilon > 0$, there is a constant $S$ such that for all $|s|\ge S$,
\[\sum\limits_{k=1}^K |\widehat {\lambda_k}(s) - \widehat {\mu_{n_k}}(s)|^2 \le \epsilon.\] 
 Now choose $\lambda_{K+1}$ and $\mu_{n_{K+1}}$ such that for all $s, |s|\le S$,
\[|\widehat {\lambda_{K+1}}(s) - \widehat {\mu_{n_{K+1}}}(s)|^2 \le \epsilon.\] 
So there is a choice of $\epsilon$ small enough so that there is a constant $C_{K+1} <\sqrt 5$ such that for all $|s| \le S$,
\[\Sigma_{K+1}(s) = \sum\limits_{k=1}^{K+1} |\widehat {\lambda_k}(s) - \widehat {\mu_{n_k,}}(s)|^2 \le C_K^2 + \epsilon\le C_{K+1}^2,\]
and for all $|s|\ge S$,
\[\Sigma_{K+1}(s) = \sum\limits_{k=1}^{K+1} |\widehat {\lambda_k}(s) - \widehat {\mu_{n_k,}}(s)|^2 \le \epsilon + 4\le C_{K+1}^2\]
This proves the induction step.
Hence, there are sequences $\lambda_k$ and $\mu_{n_k}$ such that for all $s$,
\[\sum\limits_{k=1}^\infty |\widehat {\lambda_k}(s) - \widehat {\mu_{n_k}}(s)|^2 \le 5.\]
\end{proof}

\begin{rem}\label{VARIANT}  It is useful to observe the following abstract version of the square function bound above.  Suppose we have Lebesgue measurable functions $\phi_n:\mathbb R\to \mathbb C$ are vanishing at infinity.  Suppose  we have $\phi_n$ being  {\em  locally asymptotic to $1$ at zero}; that is, for all $\epsilon$ and $M$, there exists $K$ such that if $k\ge K$, then $|\phi_k(s) - 1|\le \epsilon$ for all $|s|\ge M$.  Then there is subsequence $\phi{n_k}$ such that the same square function bound obtained above holds. That is, for some constant $C$, for any $s\in \mathbb R$, we have
$\sum\limits_{k=1}^\infty |\lambda_k(s) - \phi_{n_k}(s)|^2 \le C^2$.
\qed
\end{rem} 

Applying the Lebesgue Convergence Theorem, we have the following result due to Kostyukovsky and Olevskii~\cite{KO}as a corollary.

\begin{cor}\label{result1}  Suppose $(\mu_n)$ is a sequence of Rajchman measures such that for any $f\in L^2(\mathbb R)$,  
\begin{equation*}
	\lim\limits_{k\to \infty} \mu_{n_k}\ast f (x) = f(x) \mbox{ a.e.} 
\end{equation*}
\end{cor}
\begin{proof} Because the measures are probability measures, it is not difficult to show that there exists a sequence $(\nu_n)$ of probability measures whose supports are shrinking to $0$ such that $\|\nu_n-\mu_n\|_1\to 0$ as $n\to \infty$.  By dropping to a subsequence we can even have $\sum\limits_{n=1}^\infty \|\nu_n-\mu_n\|_1 < \infty$.  Then applying Proposition~\ref{squares}, we have a subsequence $\nu_{n_k}$ such that
for all $f\in L^1(\mathbb R)$, $\lim\limits_{k\to \infty} \nu_{n_k}\ast f = f$ a.e.  It follows immediately that                            
for all $f\in L^1(\mathbb R)$, $\lim\limits_{k\to \infty} \mu_{n_k}\ast f = f$ a.e. 
\end{proof}

\subsection{Contractions of a measure and rates}

In particular, we can consider the case of contractions of a measure.  It is not necessary, but in any case we take our measures $\mu$ to be supported in $[0,1]$.   So while we could look at Fourier transforms just on $\mathbb Z$, in the end we have to consider also the Fourier transform on $\mathbb R$ too.

The easiest way to understand what a contraction does is by looking at the Fourier transform.  Take $t > 0$ and a finite regular Borel measure $\mu$ with support a subset of $[0,1]$.  Let $C_t\mu$ be its contraction by $t$.  Then we have the Fourier transform $\widehat {C_t\mu}(s) = \widehat {\mu} (ts) = \int_{\mathbb R} \exp (-its)\, d\mu(s)$.  Without looking at any other formulas, you can see that this is a positive-definite function and so must be the Fourier transform of a probability measure too.  

But of course we have other specific formulas.  For example, 
\medskip

A)	if $f\in C_c(\mathbb R)$, then $\int f\, dC_t{\mu} = \int_{-\infty}^\infty f(ts) \, d\mu(s)$.  
\medskip

\noindent Also, 
\medskip

B)	if $A\subset \mathbb R$ is a Borel set, then $C_t{\mu} (A) =\int 1_A(ts)\, d\mu(s)= \mu (A/t)$.  
\medskip

\noindent Either of these two formulas could be used as equivalent definitions of the contraction $C_t\mu$.
\medskip

We can see immediately that $C_t{\mu}$ has support a subset of $[0,t]$.  
\medskip

We consider the convolution $f\to f\ast C_t\mu$ where $f\ast C_t\mu (y) = \int_{\mathbb R} f(y-x)\, \, dC_t\mu (x) = \int_{\mathbb R} f(y – tx)\, d\mu(x)$.  

\begin{rem} For example, if $\mu = 1_{[0,1]}\lambda$ with $\lambda$ the usual Lebesgue measure on $\mathbb R$, then for $f\in C_c(\mathbb R)$, we have $f\ast C_t\mu (y) =\int_0^1 f(y – tx)\, d\lambda(x) = \frac 1t\int_0^t f(y – w)\, d\lambda(w)$.  Hence, if $t_n\to 0^+$, we have for $f\in L^1(\mathbb R,\lambda)$, $f\ast C_{t_n}\mu (y) = f\ast (\frac 1{t_n}1_{[0,t_n]}\lambda) (y) = \frac 1{t_n}\int_0^{t_n} f(y - x)\, d\lambda(x) \to f(y)$ for a.e. $y$ with respect to $\lambda$.
\end{rem}

With the definitions of contractions above, it is easy to see that we have the following consequence of Proposition~\ref{squares}.

\begin{cor}\label{result1contract} Suppose $\mu$ is a Rajchman measure on $[0,1]$.  Then there is a sequence $\epsilon_n\to 0^+$ as $n\to \infty$ such that  such that for any $f\in L^2(\mathbb R)$,  
\begin{equation*}
	\lim\limits_{n\to \infty} C_{\epsilon_n}\mu\ast f (x) = f(x) \mbox{ a.e.} 
\end{equation*}
\end{cor} 

The proof of Proposition~\ref{squares} and Corollary~\ref{result1contract} are not explicit.  But the details can be used to make the examples explicit.  In particular, we have the following type of result.

\begin{cor}\label{explicit}  Suppose $\mu$ is a probability measure supported on $[0,1]$ that is a Rajchman measure.  Suppose $h(t)$ is decreasing with $h(t) \to 0^+$ as $|t| \to \infty$
and that $|\widehat \mu(t)| \le h(t)$ for all $|t|\ge 1$.  Then if $(t_n)$ is decreasing with $t_{n+1}\le t_n/2^nT_n$ where $T_n \ge \max (2^n,h^{-1}(1/2^n))$, we have $C_{t_n}\mu\ast f\to f$ a.e. for all $f\in L^2(\mathbb R)$.
\end{cor}
\begin{proof}  We claim that these conditions make it so that $\Sigma (t) = \sum\limits_{n=1}^\infty |\widehat \lambda (t_nt) -\widehat {\mu}(t_nt)|^2$ is uniformly bounded for $t\in \mathbb R$.  
Then the method of proof of Proposition~\ref{squares} applies and gives the result.  

Let $\Sigma_K(t) = \sum\limits_{n=1}^K |\widehat \lambda (t_nt) -\widehat {\mu}(t_nt)|^2$.  
We have the bound for the Fourier transform of the Lebesgue measure $\lambda$: $|\widehat \lambda (t)| \le 2/|t|$.  
Now take $|t|\ge T_K/t_K$.  Then $|tt_k|\ge  |tt_K| \ge T_K$ for all $k, 1\le k \le K$.  So $1/|tt_k|^2 \le 1/T_K^2$ and because
$h$ is decreasing, $h(tt_k)^2 \le h(T_K)^2$ for all $k, 1 \le k\le K$.
So for a universal constant $C$, if $|t|\ge T_K/t_K$, we can bound $\Sigma_K (t) \le \sum\limits_{k=1}^K C/|tt_k|^2 + Ch(tt_k)^2\le
C\,K/T_K^2 + C\,K\,h(T_K)^2$.  So if we take $T_K = \max (2^K, h^{-1}(1/2^K))$, we would have
$\Sigma_K(t) \le C(1/4^K)$ when $|t|\ge T_K/t_K$

Also, we can bound   $ |\widehat \lambda (t_{K+1}t) -\widehat {\mu}(t_{K+1}t)|^2 \le C \sigma_K^2$ for $|t|\le T_K/t_K$ if we take $t_{K+1} \le \sigma_K/(T_K/t_K)$.  This is because for a probability measure $\nu$ on $[0,1]$, we have $|\widehat {\nu}(t) -1|
\le \max \{|\exp (i t x) - 1|: 0\le x\le 1\} \le C|t|$.    Hence,
$ |\widehat \lambda (t_{K+1}t) -\widehat {\mu}(t_{K+1}t)|^2 = |(\widehat \lambda (t_{K+1}t) -1) -(\widehat {\mu}(t_{K+1}t)-1)|^2
\le  C|t_{K+1}t|^2\le C|\sigma_K/(T_K/t_K)t|\le C \sigma_K^2$ when $|t|\le T_K/t_K$.
But then if we take $\sigma_K = 1/2^K$, for $|t|\le T_K/t_K$, we would get a bound 
 $\Sigma_{K+1}(t) = \sum\limits_{n=1}^{K+1} |\widehat \lambda (t_nt) -\widehat {\mu}(t_nt)|^2 \le \Sigma_K(t) + \sigma_K^2 \le \Sigma_K(t) + C(1/2^K)^2$ .  
\medskip

Now without loss of generality take $t\ge 0$ and choose $K$ such that $T_{K-1}/t_{K-1} \le t \le T_K/t_K$, where for simplicity we take $t_0 =1$ and $T_0 = 0$.  
Then take $L \ge K$ and consider what bound we can get on $\Sigma_L (t)$.  
We would then have $t \le T_L/t_L$ and so $\Sigma_L(t) \le \Sigma_{L-1}(t) + C/4^{L-1}$.  Again, we would have $t\le T_{L-1}/t_{L-1}$ since $L$ is large, and so 
$\Sigma_L(t) \le \Sigma_{L-2}(t) + C/4^{L-2} + C/4^{L-1}$.  Proceeding inductively, we get
$\Sigma_L(t) \le \Sigma_{K-1}(t)+C/4^{K-1} +\dots + C/4^{L-1}$.
But then because $t \ge T_{K-1}/t_{K-1}$, we would also have $\Sigma_{K-1}(t) \le C(K-1)/4^{K-1}$.  Hence, we have 
$\Sigma_L(t)\le C(K-1)/4^{K-1} +C/4^{K-1}+\dots + C/4^{L-1}$.   
Hence, for all $t$, $\Sigma(t) \le C \sum\limits_{k=1}^\infty k/4^k$.  
\end{proof}

\begin{rem}
It is an interesting issue to ask what rate $t_n$ will work in general.  For Lebesgue measure, anything going to zero will work.  So  any measure absolutely continuous with Lebesgue with a bounded density will work too (by bounding the maximal function).  Maybe there are examples that this is true even though the density is not bounded.  

But of course it leaves open this issue: can we get a singular measure (say Rajchman) such that for any $t_n\to 0$, we have $C{t_n}\mu\ast f\to f$ a.e. for all $f\in L^2$?  Perhaps this is too liberal and we should at least take $t_n\le 1/n$, or something like that.  Or, if $\mu$ is singular is this not ever possible, the $t_n$ always having to increase rather rapidly?!  

The construction in the attached is giving $t_n$ but the ratio $t_{n+1}/t_n$ is going to zero at least exponentially in order to get the estimates to work out.  
So it is reasonable to ask if there are singular measure $\mu$ so that contracting by a lacunary rate $1/2^n$ will give a pointwise good approximate identity?  
\qed
\end{rem}

\section{Riesz Product Measures}\label{productmeasures}

In this section, our goal is to construct a non-Rajchman probability measure on $[0,1]$ and a sequence $(t_n)$ with $\lim\limits_{n\to \infty} t_n = 0$, such that the contractions $C_{t_n}\mu$ have $C_{t_n}\mu\ast f \to f$ a.e. as $n\to \infty$ for all $f$ in some Lebesgue class.  This requires having some specific information about the behavior of the Fourier transform of the non-Rajchman measure $\mu$.  But while there is a large literature surrounding non-Rajchman singular measures, specifics about their Fourier transforms can often be difficult to obtain.

Perhaps fittingly, we rely on one of the earliest examples of a non-Rajchman measure, first introduced by Fr\'ed\'eric Riesz over 100 years ago \cite{Riesz}: the Riesz product measures. These measures, each the weak* limit of a sequence of products, offer a great deal of control over their Fourier transforms.  So this is the approach we take here. Though Zygmund~\cite{Zygmund} provides a classical view on the basics of Riesz products, there are many other good references.  For example, the text by Graham and McGehee~\cite{GM} gives an excellent introduction to the subject. 

A Riesz product is an infinite product of the form 
\begin{equation*}
	\prod_{m=1}^{\infty} \left( 1+ a_m \cos(2 \pi n_m x) \right),
\end{equation*}
where $|a_m| \le 1$, $a_m \neq 0$, for all $m$ and $(n_m)$ is taken to be sufficiently lacunary, for example, $n_{m+1}/n_m \ge \Delta \geq 3$. For the purposes of our construction, the coefficients $a_m$ can be chosen so that $a_m=1$ for all $m$. 

Note that the partial products are trigonometric polynomials. Consider the partial product
\begin{equation*}
	\prod_{m=1}^{k} \left( 1 + \cos(2\pi n_{m} x) \right).
\end{equation*}
If $\ell_k = \sum_{m=1}^k n_{m}$, we then have that 
\begin{equation*}
	p_k(x) = 1 + \sum_{j=1}^{\ell_k} c_j \cos(2\pi j x) = \prod_{m=1}^{k} \left( 1 + \cos(2\pi n_{m} x) \right),
\end{equation*}
with $c_j$ non-zero if, and only if, there is some $k$ such that  
\begin{equation}
	j= n_k + \sum_{i=0}^{k - 1} s_i n_i, \label{jay}
\end{equation}
where $s_i \in \{ -1, 0, 1\}$.

The Riesz product measure is then the weak* limit of the measures given by these polynomials: for each Borel set $A\subset [0,1]$, 
\begin{equation*}
	\mu(A) = \lim_{k \to \infty} \int_A p_k(x) \,\, d\lambda(x).
\end{equation*}

As a consequence, the Fourier coefficients of this measure $\widehat{\mu}(j)$ are non-zero only where $j$ is of the same form as in (\ref{jay}).

So, by appropriate choice of the sequence $(n_m)$, we may construct a measure whose non-zero Fourier coefficients on the integers exist entirely in disjoint \emph{clouds} of integers, each an interval of the form 
\begin{equation*}
	\left[n_k - \sum_{m=1}^{k-1} n_m, n_k + \sum_{m=1}^{k-1} n_m \right].
\end{equation*}
Indeed, by choosing a sufficient rate of growth for $n_m$, we can ensure that each successive cloud is as far away as we might like from its predecessor.

The coefficients $\widehat{\mu}(j)$ are not tending to zero: in fact, $\widehat{\mu}(j) = 1/2$ whenever $j$ is an element of the sequence $(n_m)$. Yet the increasing gaps between each cloud and the rapid decay of the coefficients on each cloud ensure that the coefficients tend to zero in mean. Thus the resulting measure is a continuous, but not Rajchman, measure. 

\section{Contractions of a Riesz Product}\label{ContractRP}

The utility of the Riesz product measure usually lies in its well-understood Fourier \emph{coefficients} given by its Fourier transform on the integers.   But at this stage we know little about the value of its Fourier transform except on the integers. 

Consider the Fourier transform of a given Riesz product measure:
\begin{equation*}
\widehat{\mu}(s) = \lim_k \int_\R e^{-2 \pi i s x} \1_{[0,1]}(x) p_k(x) \de \lambda(x).
\end{equation*}

We wish to show that, outside of the clouds, we can arrange for this Fourier transform value to be small.  To do so, we will need to make stipulations regarding the rate of growth of $(n_k)$. As our later construction relies on taking a subsequence of a sequence that grows sufficiently quickly, this does not run counter to our purposes. Let $\Delta \geq 4$.

With this, note that we may slightly widen each cloud. Define
\begin{equation*}
	C_k = \left[ \frac{1}{2} n_k, \frac{3}{2} n_k \right].
\end{equation*}
As a consequence of the growth rate of $n_k$, each cloud is contained in its corresponding $C_k$. In fact, for $\Delta \geq 4$, we have that 
\begin{equation}
	  [n_k - \sum_{j=1}^{k-1} n_j, n_k + \sum_{j=1}^{k-1} n_j] \subset  \left[\frac23 n_k , \frac43 n_k \right]  \subset C_k. \label{distancegap}
\end{equation}
The gap between the cloud centered at $n_k$ and that centered at $n_{k+1}$ must then contain the interval
\begin{equation*}
	G_k = \left( \frac{3}{2} n_k, \frac{1}{2} n_{k+1}  \right) .
\end{equation*}

Now, 
\begin{align}
	|\widehat{\mu}(s)| &= \lim_k |\int_{0}^1 e^{-2 \pi i s x} p_k(x) \de \lambda(x)| \nonumber \\
					& = \lim_k \left|\int_{0}^1 e^{-2 \pi i s x} \left( 1 + \sum_{j=1}^{\ell_k} c_j \cos(2\pi j x)  \right) \de \lambda(x) \right| \nonumber \\
					&\leq \lim_k \sum_{j=1}^{\ell_k} \left| c_j  \int_{0}^1 e^{-2 \pi i s x}    \cos(2\pi j x)   \de \lambda(x) \right|  \nonumber \\
					&=	\lim_k \sum_{j=1}^{\ell_k} c_j  \left| \frac{ e^{-2 \pi i s }\left(j \sin(2\pi j) - is \cos(2\pi j) \right)  +is }{2 \pi (j-s) (j+s) } \right| \nonumber	\\
					&\leq 	\lim_k \sum_{j=1}^{\ell_k} c_j  \left( \frac{|j + s|}{2 \pi |(j-s) (j+s) |}  + \frac{|s|}{2 \pi |(j-s) (j+s) |} \right) \nonumber \\
					&= 	\sum_{j=1}^{\infty} c_j  \frac{1}{2 \pi |j-s|}. \label{bound1}
\end{align}

Recall that $c_j \neq 0$ only when $j = n_k + \sum_{i=1}^{k - 1} s_i n_i$ for some $k$. Since these are all contained in the clouds, the size of $\mu(s)$ may be thought of as depending on the distance of $s$ from the all the sets $C_k$. We now restrict our attention to $s>0$.

\begin{rem}
The inclusion of the characteristic function $\1_{[0,1]}$ when defining the measure on $\R$ eliminates much of the symmetry we may have enjoyed on $\T$. Yet, while $\mu(-s)$ may not equal $\mu(s)$, (\ref{bound1}) remains, and the negative values are even further away from the clouds than the positive ones are. Thus, $|\mu(s)| \leq |\mu(|s|)|$.  
\end{rem}

Suppose that $s \in G_m$. By (\ref{distancegap}), the distance from $s$ to each $C_k$ to the left of $s$ is always more than $\frac{3}{2} n_m - \frac{4}{3} n_k$, while the distance to those $C_k$ to the right of $s$ is greater than $\frac{2}{3} n_k - \frac{1}{2} n_{m+1}$. Noting that each $C_k$ contains $3^{k-1}$ elements $j$ for which $c_j \neq 0$, we have 
\begin{align*}
\sum_{j=1}^{\infty} c_j  \frac{1}{2 \pi |j-s|} &= \sum_{k=1}^{\infty} \sum_{j \in C_k} c_j  \frac{1}{2 \pi |j-s|} \\
& < \sum_{k=1}^m   \frac{3^{k-1}}{ \frac{3}{2} n_m - \frac{4}{3} n_k  } + \sum_{k=m+1}^\infty  \frac{3^{k-1}}{ \frac{2}{3} n_k - \frac{1}{2} n_{m+1}} \\
&\leq \frac{m3^{m-1}}{ \frac{1}{6} n_m} +  \sum_{k=m+1}^\infty  \frac{3^{k-1}}{ \frac{1}{6} n_k }. 
\end{align*}

Since $n_k$ must grow at least as fast as $4^k$, we have that, for all $s \in G_m$, 
\begin{equation*}
	|\mu(s)| \leq B_m = 6(m+C)\frac{3^{m-1}}{n_m},
\end{equation*}
for some constant $C$.   So actually besides $B_m\to 0$ as $m\to \infty$, we have $B_m$ decreasing (at least) exponentially and $\sum_{m=1}^\infty B_m < \infty$.

\section{Construction of a Good Continuous, Non-Rajchman Measure}\label{GOOD}

\begin{thm}There is a Riesz product measure $\mu$ and a decreasing positive sequence $(t_k)$ with $\lim t_k =0$ yielding a sequence of contractions $(\mu_k)$ so that the square function
		\begin{equation*}
		Sf = \left\{ \sum_{k \in \N} | C_{t_k}\lambda \ast f - C_{t_k}\mu \ast f |^2 \right\}^\frac12 
	\end{equation*}
	has a strong $L^2$ bound:
	\begin{equation*}
		\|Sf\|_2 \le C\|f\|_2  \mbox{ for all } f\in L^2(\mathbb R,\lambda).
	\end{equation*} 
\end{thm}

\begin{rem}
	There appear to be few obstacles to stating this result in a different way, perhaps in greater generality. Suppose we began with any Riesz product measure satisfying the necessary growth conditions for the Fourier estimates from Section \ref{ContractRP} to hold. The arguments below would then require us to refine this measure through selection of a subsequence of $(n_m)$. This must be done in order to ensure the disjointness of the clouds. Given the need to manipulate the placement of the non-zero coefficients, these arguments do not reveal whether the conclusion holds generally; that is, whether there is, for each Riesz product measure, a sequence $(t_k)$ so that the resulting square function has a strong $L^2$ bound.   
\end{rem}

\begin{proof}
We consider $\mu$ and $\lambda$ as measures on $\R$ with support contained in $[0,1]$. 
	
Our goal is to adapt the framework from Section \ref{result1} to construct a Riesz product measure $\mu$ and a sequence $(t_k)$, with $\lim t_k = 0$, so that 
\begin{equation*}
	\mathcal S (s) = \sum_{k \in \N} | \widehat{\lambda}(t_ks)  - \widehat{\mu}(t_ks) |^2 \le C 
\end{equation*}
for some constant, $C$. As before, we then may conclude that the corresponding square function has a strong $L^2$ bound. 

Since the clouds are contained in the positive half of the real line, we will focus on $s \geq 0$. The construction will be symmetric on the negative reals, where, in any case, we have that $|\widehat{\mu}(s)| \leq |\widehat{\mu}(|s|)|$.

One method for accomplishing this is to intertwine the images of the clouds under dilation by $\rho_k = \frac{1}{t_k}$. Let $C_m$ denote the interval containing the cloud centered at $n_m$, as in Section \ref{productmeasures}, and let $\rho_0 =1$.  We also use the following notation:  $A,B\subset \mathbb R$ have $A < B$ if for all $x\in A$ and $y\in B$, we have $x < y$.   We use the fact that a cloud $C_m =[a_m,b_m]$ is an interval about $n_m$ with $b_m/a_m\le C$ for some constant.  This is necessary for the following induction.

We first choose $n_1$, which fixes $C_1$.  Then choose $t_1$ so that $C_1 = \rho_0C_1 < \rho_1C_1$.  Next we choose $n_2 > n_1$ so that $\rho_1C_1 < \rho_0C_2 < \rho_1C_2$, and then $t_2 < t_1$ so that $\rho_1C_2 < \rho_2C_0 < \rho_2C_1 < \rho_2C_2$.    It is worth observing here, since it also occurs in the inductive steps, that while ratios may remain the same, these choices force the distance between successive clouds to increase.  Also, with each choice of $n_m$ subsequently, increasing the distance between clouds decreases the width of the cloud $C_m$ relative to the surrounding gap.  

Continue this inductively. As $t_k$ decreases, $\rho_k$ increases. At the $k^{\mbox{th}}$ step, we choose $n_{k+1}$ so that we have 
\begin{equation*}
\rho_{k-1}C_k < \rho_kC_0 <\rho_kC_1 <\cdots < \rho_kC_k < \rho_0C_{k+1} < \rho_1C_{k+1} <\cdots < \rho_kC_{k+1}.
\end{equation*}
It follows that for any distinct $k,l$, and any $m_1,m_2$, only one of the following occurs:  
\begin{equation*}
\rho_kC_{m_1} < \rho_lC_{m_2}, \, \rho_lC_{m_1} < \rho_kC_{m_2},  \, \rho_kC_{m_2} < \rho_lC_{m_1}, \mbox{ or } \rho_lC_{m_2} < \rho_kC_{m_1}.  
\end{equation*}

This gives the disjointness that we wanted; thus $t_ks$ is in, at most, one cloud for all $k \in \N$. 

We now use the observation in Remark~\ref{VARIANT}.  First, take the sequence $(t_n)$ constructed above.  Then we write $\widehat {C_{t_n}\mu} = \phi_n + \sigma_n$
where $\sigma_n(s) = \sum\limits_{m=1}^\infty \widehat \mu(t_ns) 1_{C_m}(t_ns)$ and $\phi_n =  \widehat \mu(t_ns) 1_{C_m}(t_ns) - \sigma_n$.  Then $\phi_n$ are Lebesgue measurable and 
vanishing at infinity.   We have automatically the required locally asymptotic to $1$ at zero because $\widehat {C_{t_n}\mu}$ has this property and $\sigma_n$ is zero in a neighborhood of $0$.
\end{proof}

\begin{rem}
	The intertwining argument is not the only way in which we might approach the construction of the Riesz product and associated contraction sequence $(t_k)$. Consider the following, more concrete, construction.
	
	Once more, we look at only $s>0$ and wish to limit the number of $k$ for which $\widehat{\mu}(t_ks)$ is large. This can only occur where $t_ks \in \cup_{m=1}^\infty C_m$. The goal is to construct an $n_m$ and $t_k$ so that for each $s$ this can occur for only a few $k$. 
	
	Let $t_k = 2^{-b_k}$ and $n_m = 2^{a_m}$, where $(b_k)$ and $(n_m)$ are unbounded positive increasing sequences which we will determine later. 

	$t_ks \in C_m$ only if 
	\begin{equation*}
		\frac{1}{2} n_m \leq t_ks \leq 2n_m.
	\end{equation*}
	Equivalently, 
	\begin{equation*}
	 -1	+ a_m  \leq -b_k + \log_2 s \leq 1 +  a_m
	\end{equation*}

	Suppose we have distinct $k_1$ and $k_2$ such that
	\begin{align*}
		-1	+ a_{m_1} &\leq -b_{k_1} + \log_2 s \leq 1 +  a_{m_1} \mbox{ and}\\
		-1	+ a_{m_2} &\leq -b_{k_2} + \log_2 s \leq 1 +  a_{m_2}.
	\end{align*}
 So multiplying the first of these inequalities by $-1$ and adding that to the second of these inequalities, we would
have 
	\begin{equation*}
		a_{m_2}- a_{m_1} = b_{k_1} -b_{k_2}.
	\end{equation*}
	and since $k_1 \neq k_2$, we must also have $m_2 \neq m_1$. 

	Let $b_k = 2^{2k+1}$ and $a_m = 2^{2m}$. Then there are an even number of twos in the prime factorization of $a_{m_2}- a_{m_1}$ while there are an odd number of twos in that of $b_{k_1} -b_{k_2}$. Thus they cannot be equal, and there are no distinct $k_1$ and $k_2$ so that $t_{k_1}s$ and $t_{k_2}s$ are both in $\cup_{m \in \N} C_m$. 

Even without taking yet another subsequence, we have the necessary subsequence bound.  See Proposition~\ref{explicit} and follow this line of reasoning.
So with a initial choice of 
	\begin{equation*}
		n_m = 2^{2^{2m}} \mbox{, and } t_k = \frac{1}{2^{2^{2k+1}}},
	\end{equation*}
	we can have the Riesz product measure yielding a pointwise a.e. good sequence of contraction operators $C_{t_k}\mu$.
\end{rem}

As a corollary, we have the following result.
\begin{prop}\label{RieszAI} There is a Riesz product measure $\mu$ on $[0,1]$, which is continuous and non-Rajchman,  and there is a sequence $t_k\to 0$ such that for all $f\in L^2(\mathbb R)$, the convolution of $f$ with the sequence of contractions, $C_{t_k}\mu$, corresponding to $(t_k)$ converges pointwise a.e. That is,
	\begin{equation*}
		\lim_k	C_{t_k}\mu \ast f = f \mbox{ a.e.} 
	\end{equation*}	
\end{prop}

\subsection {An Alternative Approach for the Fourier Transforms off Clouds}

As we have discussed in the previous section, one difficulty with the Fourier transform on $\mathbb R$ of a Riesz product is that we do not know precisely its behavior on $\mathbb R$.  Instead of carrying out the analysis above, we could just stick what we know about the Fourier transform on the integers and use the theorems in Goldberg~\cite{Goldberg}. We take the Riesz product construction as above giving us $\mu$.  We can extend this periodically to $\mathbb R$, multiply by $(1 - \cos (2\pi t))/t^2$ and normalize the result to get a new measure $\nu$ on $\mathbb R$ with these properties.  First, $\widehat {\nu} = \widehat {\mu}$ at all integer values.  But also, $\widehat {\nu}$ being interpolated from $\widehat \mu$ on $\mathbb Z$ has its Fourier transform equal to zero on all intervals $[k,k+1]$ disjoint from the intervals $C_m$ above.

 Take $L$ to be the Fourier transform of $\ell = 1_{[0,1]}\lambda$. 
\begin{prop} There is a constant $C$ such that the square function 
\[Sf = \left (\sum\limits_{k=1}^\infty |C_{t_k}(\nu)\ast f - C_{t_k}(\ell)\ast f|^2\right )^{1/2}\] 
satisfies $\|Sf\|_2 \le C\|f\|_2$ for all $f\in L^2(\mathbb R)$.
\end{prop}
\begin{proof} We take Fourier transforms and show that there is a uniform bound on the associated square function multiplier
\[\sum\limits_{k=1}^\infty |\widehat {\nu}(t_kx) - L(t_kx)|^2 \le C^2\]
for all $x\in \mathbb R$.  We use the facts proved above about our choices of $n_m$ and $t_k$ to get this estimate.  
\end{proof}

\begin{cor}\label{strong} There is a constant $C$ and a maximal function inequality $\|\sup\limits_{k\ge 1} |C_{t_k}(\nu)\ast f|\|_2 \le C\|f\|_2$ for all $f\in L_2(\mathbb R)$.  Moreover, for all $f\in L_2(\mathbb R)$, as $k\to \infty$,  $C_{t_k}(\nu)\ast f \to f$ a.e.
\end{cor}

Now consider $\nu$ restricted to $[0,1]$.  We can take a small $\epsilon$ and consider $J = [0,1-\epsilon]$ and consider the measure $\mu_0 = c1_J\mu$, with $c$ chosen so that $\mu_0$ is a probability measure.  Because $\epsilon$ is small, we would have $\mu_0$ a continuous measure whose Fourier transform on $\mathbb Z$ does not go to zero.  But moreover, there is some constant $C$ such that $C_{t_k}(\mu_0) \le C\, C_{t_k}(1_J\nu) \le C\, C_{t_k}(\nu)$.  Hence, from Corollary~\ref{strong}, we get also

\begin{prop}   There is a constant $C$ and a maximal function inequality $\|\sup_k |C_{t_k}(\mu_0)\ast f|\|_2 \le C\|f\|_2$ for all $f\in L_2[0,1]$.  Moreover, for all $f\in L^2[0,1]$, as $k\to \infty$, $C_{t_k}(\mu_0)\ast f \to f$ a.e..
\end{prop}

\begin{cor} \label{Alternate} There is a continuous probability measure $\mu_0$ on $[0,1]$, which is not a Rajchman measure, for which there is an approximate identity given by the contractions $C_{t_k}(\mu_0)$ , so that 
 for all $f\in L^2[0,1]$, as $k\to \infty$, $C_{t_k}(\mu_0)\ast f \to f$ a.e.
\end{cor}

\begin{rem}   It is possible to use this technique and an adjustment of the multiplier to derive the conclusion of Corollary~\ref{Alternate} for a Riesz product measure.
\qed
\end{rem}
\section{ A Construction of a Bad Continuous, Non-Rajchman Measure}\label{BAD}

Fix $a < 1$.  The basic idea is to construct a probability measure $\mu$ so that for the contractions by any subsequence of $a^l$ are an approximate identity that fail a.e. convergence on a Lebesgue class, perhaps even $L^\infty(\mathbb R)$.   

First take the interval $J = [b,1]$ with $a < b < 1$.  For notation later, we will
want to have $J_0 = J$.  Take $J_l = a^lJ, l \ge 1$.  These are pairwise disjoint because
$a^{l+1} < ba^l  $ for all $l\ge 1$.  

\begin{prop}\label{KRONECKER} There is a compact, perfect set $K \subset J$ such that for all $L$, the compact set $\bigcup\limits_{l=1}^L a^l K$ is a Kronecker set.
\end{prop}
\begin{proof}  The properties that we want for $K$ can be arranged using the type of construction of Kronecker sets given in Rudin~\cite{Rudin2}, Section 5.2.  In fact, the most straightforward approach is to use the construction in Section 5.2 to build a compact, perfect, totally-disconnected set $K\subset J$ that is sufficiently independent of the values of $a^l, l=1,\dots,L$ so that the Kronecker property holds. 
\end{proof}

Now take a continuous probability measure $\mu$ supported on $K$ in Proposition~\ref{KRONECKER}.  Using the entropy method in Bourgain~\cite{Bourgain}, as in
Kostyukovsky and Olevskii\cite{KO}, we have the following.

\begin{prop}\label{BasicBAD}  For any sequence $l_m\to \infty$ as $m\to \infty$, the contractions $C_{a^{l_m}}\mu$ are an approximate identity for which pointwise convergence of $C_{a^{l_m}}\mu \ast f$ fails to hold on any $L^p[0,1], 1 \le p \le \infty$.  Indeed, there are characteristic functions $f$ for which the pointwise convergence fails.
\end{prop}
\begin{proof}  We have the measure $C_{a^{l_m}}\mu$ supported on disjoint subsets of $a^{l_m}K$.  Since any finite union of these sets are Kronecker sets, the entropy method of \cite{Bourgain} applies.
\end{proof}

\begin{rem}  The construction really does not use the full strength of the Kronecker condition.  Indeed, one only needs to approximate by characters any continuous function $\phi$ taking the values $\pm 1$ on $\bigcup\limits_{m=1}^M a^{l_m}K$.  Can we build compact totally-disconnected perfect sets $K$ that are not Kronecker sets, but have this more limited approximation property?  Note: it is no doubt clear, but it does no harm to point out that these unions are compact, totally-disconnected sets, so having functions that are continuous and take the values $\pm 1$ is not a contradiction as it would be on a domain like $[0,1]$.  
\qed
\end{rem}

\noindent {\bf Question}:  Can we extend the construction of Proposition~\ref{BasicBAD} so that there for all sequences $t_m\to 0$ the contractions $C_{t_m}\mu$ are not an a.e. pointwise good approximate identity, say on at least $L^\infty(\mathbb R)$?

\section{Various Questions}\label{QUESTIONS}

The results in this article suggest a number of questions concerning the behavior of approximate identities in general, and contractions of a fixed measure in particular.  we return to the types of questions first mentioned in the Introduction~\ref{intro}.  In short, to obtain a.e. convergence of convolutions by the approximate identity, what do we need to know about the measures and the functions?

Perhaps the most basic question is this

\begin{quest}\label{Q1}  For which continuous, non-Rajchman probability measures $\mu$ is there a sequence $C_{t_k}\mu$ of contractions such that for some Lebesgue space $L^p(\mathbb R)$, we have $C_{t_k}\mu \ast f \to f$ a.e. for all $f\in L^p(\mathbb R)$?
\end{quest}

\begin{rem}  If there is a good answer to Question~\ref{Q1}, it would seem to be likely to be about the analytical and/or combinatorial nature of the sets $\{k: |\widehat {\mu}(k)| \ge \delta\}$ as we vary $\delta > 0$.
\qed
\end{rem}

But more specifically, we might consider the following questions about Riesz products.

\begin{quest}\label{Q2}
	Suppose $\mu$ is a non-Rajchman Riesz product measure. Is there a $(t_k)$ so that 
	\begin{equation*}
\lim\limits_{k\to \infty} C_{t_k}\mu\ast f = f \mbox{ a.e. for all } f\in L^2(\mathbb R)?
	\end{equation*}
\end{quest}

We have seen that there are measures with with a sequence of contractions such that all subsequences are not behaving well pointwise a.e.  However, the construction is far from the Riesz product construction.  So this question is perhaps reasonable.

\begin{quest}\label{Q3}
	Is there a non-Rajchman  Riesz product measure so that 
	\begin{equation*}
		\lim\limits_{k\to \infty} C_{t_k}\mu \ast f = f \mbox{ a.e. for all } f\in L^2[0,1]?
	\end{equation*}
	for all sequences $(t_k)$ that grow faster than a certain fixed rate? 
\end{quest}

\begin{rem}  The construction that was given for a good Riesz product uses a rapidly increasing sequence $(t_k)$ with additional structural information.  So this question also seems to be a reasonable one.
\qed
\end{rem}

\begin{quest}\label{Q4}
	Is there an example of a sequence of contractions $C_{t_k}\mu$ of a non-Rajchman Riesz product measure $\mu$ where the maximal functions $\sup\limits_{k\ge 1} |C_{t_k}\mu\ast f|$ exhibit a strong $L^p(\mathbb R)$ bound for $1 < p < 2$, or even a weak $L^1(\mathbb R)$ bound? If so, what conditions must there be  on the $(t_k)$ and $\mu$?
\end{quest}

\begin{rem}  It seems unlikely that for a Riesz product construction, we can give specifics about which $L^p$-bounds one can have, for example where there is a strong $L^2$ bound but failure of pointwise convergence on $L^q$ for $q < 2$.  Nonetheless, this is not an unreasonable question.
\qed
\end{rem}

There is always the option of replacing the sequence of contractions $C_{t_n}\mu$ by an appropriate sequence of averages of these.  For example, we could take the Ces\`aro averages $\alpha_N = \frac 1N \sum\limits_{n=1}^N C_{t_n}\mu$.  The example we give of a measure and a sequence of contractions such that all subsequences are bad can be seen to be bad also for any sequence of averages along a subsequence of $C_{t_n}\mu$ also.  But in any case, we can reasonably ask for the conditions, say on a Riesz product construction, on $\mu$ such that the following holds:

\begin{quest}\label{Q5}  When can we have the averages $\alpha_N$ pointwise good on $L^2(\mathbb R)$, for example in the case that the underlying measure $\mu$ is a Riesz product measure? 
\end{quest} 

\section {\bf Appendix: Details about Convolution Operators}\label{details}

There are some technical issues to consider.  Most of what is outlined here is also carried through in Hewitt and Ross~\cite{HR1}, Section 20 (with an Addendum to Volume I in Hewitt and Ross~\cite{HR2}).  They handle the  issues for general locally compact groups, but we actually will focus on the very special classical case of $\mathbb T$ with the usual normalized Lebesgue measure.  Indeed, we will consider this as $[0,1]$ with addition module one, and the usual normalized Lebesgue measure $m$.  Sometimes it is convenient to actually not use addition modulo one, but consider the operators on $L^1(\mathbb R,m)$ with $m$ again being the usually Lebesgue measure now on $\mathbb R$.  But nonetheless, there are still some pitfalls.  It is interesting that one of them was an issue for Hewitt and Ross themselves, the pitfall of checking joint measurability before using theorems on multiple integrals.  This is the reason that they added the Addendum to Volume I in \cite{HR2}.

We take a positive, finite Borel measure $\mu$ on $[0,1]$.  We want to explain how this gives a bounded linear operator on $L^1[0,1]$ using the formula $\mu\ast f(x) = \int f(x-y)\, d\mu(y)$, for $f\in L^1([0,1],m)$.  We want this to apply in particular to singular measures $\mu$.  The difficulty is that we need to make sense of this formula given that $f$is actually only an equivalence class of functions, even though by tradition the notation suggests it is just a single function.  We cannot just substitute for $f$ in the formula some $g=f$ a.e.   Indeed, if $\mu$ is supported on a Lebesgue null set $E$, there are functions $g=f$ a.e. such that $g(-y)$ restricted to $E$ is not a Borel measurable function.  Hence, this formula would not make sense for $x=0$.  In fact, it could be arranged that for any function $h$ on $E$, there is $g=f$ a.e. such that $g(-y)=h(y)$ on $E$.  

For notation, we let $\mathcal B_k$ be the Borel measurable sets in $[0,1]^k$.     We also let $\mathcal L_k$ be the Lebesgue measurable sets in $[0,1]^k$.  Here $k=1,2$.  As usual, we say that a function $g$ is $\Sigma$-measurable, with respect to some $\sigma$-algebra $\Sigma$, when $f^{-1}(U) \in \Sigma$ for all open sets $U$.  

We first consider functions $f \ge 0$ that are $\mathcal B_1$-measurable.  As usual, this means that $f^{-1}(U) \in \mathcal B_1$ for all open sets $U$.  We define $T_x$ by $T_x(y) 
=x - y$ for all $y$.

\begin{prop}\label{meas} Given $f\ge 0$ that is $\mathcal B_1$-measurable, for all $x$, the function $F(y) = f(x-y) = f\circ T(y)$ is $\mathcal B_1$-measurable too.   So the integral $\int f(x-y) \, d\mu(y) $ is well defined in $[0,\infty]$.  
\end{prop}
\begin{proof}
Now for any open set $U$, we have $(f\circ T)^{-1}(U) = T^{-1}(f^{-1}(U))$.  Hence, what we need to show is that if $B\in \mathcal B_1$ and $T$ is continuous, then $T^{-1}(B)$ is in $\mathcal B_1$.  We would then apply this to $B = f^{-1}(U)$.  But it is easy to check that the class of sets $\mathcal B$ consisting of sets $B$ such that $T^{-1}(B)$ is in $\mathcal B_1$, is actually a $\sigma$-algebra.   Since $T$ is continuous, $\mathcal B$ contains all open sets.  Hence,  $\mathcal B_1$ is a subset of $\mathcal B$.  That is, for all Borel sets $B$, the set $T^{-1}(B)$ is a Borel set, which is what we wanted to show.   
\end{proof}

But we want more, actually to show how the convolution $\mu\ast f$ gives a bounded operator on $L^1[0,1]$.  We can take the viewpoint used in Hewitt and Ross~\cite{HR1} that the equivalence class for $f\in L^1[0,1]$ is just represented by a $\mathcal B_1$-measurable function in that equivalence class.  But we can also use the extension of $\mu$ to include $\mu$-null sets and resolve problems with definitions that way, as we will see.

In any case, we have to consider joint measurability of functions of two variables in order to make sense of iterated integrals like $\int \int f(x-y) \, d\mu(y)\, dm(x)$.  We first see how

\begin{prop}\label{jtmeas} Suppose $f\ge 0$ is $\mathcal B_1$-measurable.  Then the positive function $G(x,y) = f(x - y)$ is measurable with respect to the product $\sigma$-algebra $\mathcal P$ for the product measure  $\mu \times m$.  
\end{prop}
\begin{proof}  Here $\mathcal P$ is the product $\sigma$-algebra $\mathcal B_1 \times \mathcal L_1$.  An argument just like the one in Proposition~\ref{meas} shows that $G:[0,1]^2 \to \mathbb R$ is $\mathcal B_2$-measurable.  So we need to show that $\mathcal B_2 \subset  \mathcal P$.    This is clear if we can show that the product $\sigma$-algebra $\mathcal S$ generated rectangles $A\times B$, where $A,B \in \mathcal B_1$, contains $\mathcal B_2$.  Then we use the obvious fact that $\mathcal S \subset \mathcal P$.  

But all open sets are in $\mathcal S$, because any open set is a countable union of rectangles $U\times V$ where $U,V$ are open sets.  So $\mathcal B_2$ being the smallest $\sigma$-algebra containing open sets tells us that $\mathcal B_2 \subset \mathcal S$. 
\end{proof}

\begin{rem}
In fact,  $\mathcal S = \mathcal B_2$.  Given the detail in the proof of Proposition~\ref{jtmeas}, what we need to show is that $\mathcal S \subset \mathcal B_2$.  Since $\mathcal S$ is the smallest $\sigma$-algebra containing sets $A\times B$ where $A,B\in \mathcal B_1$, it suffices to shows that if $A,B$ are in $\mathcal B_1$, then $A\times B$ is in $\mathcal B_2$.  But $A\times B = A\times [0,1]\cap [0,1]\times B$.  Moreover, the class $\{A: A\times [0,1]\in\mathcal B_2\}$ contains the open sets and is a $\sigma$-algebra.  The same is true for $\{B:[0,1]\times B\in\mathcal B_2\}$.  So both of these $\sigma$-algebras contain $\mathcal B_1$ and hence $A\times [0,1]$ and $[0,1]\times B\in \mathcal B_2$ for any $A,B \in \mathcal B_1$.  This type of result can be proved in great generality.  For example, see Bogachev~\cite{BOG2}, Lemma 6.4.1 and Lemma 6.4.2.  
\qed
\end{rem}

Now we can apply the Fubini-Tonelli Theorem.   

\begin{prop}\label{Convolve} Given a positive, Borel probability measure $\mu$ on $[0,1]$, and a $\mathcal B_1$-measurable, $m$-integrable function $f$, $O(f) = f\to \mu\ast f$ is a well defined real number for a.e. $x [m]$, and $O(f)$ is $m$-integrable with
$\int |O(f)| \, dm = \int |\int f(x-y)|\, d\mu(y)|\, dm(x) \le \int |f|\, dm$.
\end{prop}
\begin{proof} Assume first that $f\ge 0$ is $\mathcal B_1$-measurable and $\int f\, dm < \infty$.   The measurability result Proposition~\ref{jtmeas} shows that we can use the Tonelli Theorem to compute 
\[\int \int f(x-y) \, d\mu(y)\,dm(x) =\int \int f(x-y) \, dm(x)\,d\mu(y) = \mu[0,1]\int f\, dm.\]
Hence, $\int \int f(x-y) \, d\mu(y)\,dm(x)$ is finite.  In addition, for a.e. $x [m]$, 
$O(f)(x) =\mu\ast f(x) =  \int f(x-y) \, d\mu(y)$ is finite and $\int O(f) dm = \mu[0,1]\int f \, dm$. 

But then if we allow $f$ to be $\mathcal B_1$-measurable and $\int |f|\, dm < \infty$, we can apply the above to $|f|$.  Then again the measurability result in Proposition~\ref{jtmeas} and the Fubini-Tonelli Theorem gives this proposition.
\end{proof}

Proposition~\ref{Convolve} says that we can define $O(f)$ for $f\in L^[0,1]$ as follows.

 \begin{defn}\label{OI} Take any Borel measurable function $g$ in the equivalence class for $f\in L^[0,1]$ and let $O(f)$ be the equivalence class in $L^[0,1]$ of  $O(g)$ with $O(g) = \int f(x-my)\, d\mu(y)$ as defined above.  
 \qed
 \end{defn}
 
\noindent The new operator $O$ is a well-defined, bounded linear operator on $L^1[0,1]$.   Indeed, if $g_1,g_2$ are both Borel measurable in the equivalence class of $f$, then $\int |O(g_1) - O(g_2)|\, dm = \int |O(g_1-g_2)|\, dm \le \int |g_1 - g_2|\, dm = 0$.  Also, $\|O(f)\|_1 \le \|f\|_1$, so $O$ is a contraction operator.

But we can also use completion of measures to define $O$ on $L^1[0,1]$.  Given a positive, finite  measure $\nu$ on a $\sigma$-algebra $\mathcal M$, we denote its completion by $\nu^*$.  This is the same as $\nu$ but extended to the $\sigma$-algebra given by adjoining all subsets of $\nu$-null sets to $\mathcal M$.  It is not enough to use just $\mu^*$, we actually need to also use $(\mu^*\times m)^*$,the completion of the product measure of the two complete measures $\mu^*$ and $m$.

\begin{prop}\label{complete} Let $f$ be a $\mathcal L_1$-measurable, $m$-integrable function.  Then
$G(x,y) = f(x-y)$ is $(\mu^*\times m)^*$-measurable.  Moreover, for a.e. $x [m]$, the function $y\to f(x-y)$ is $\mu^*$-measurable, and $O(f) = \mu^*\ast f$ is well-defined and finite for a.e. $x$ with respect to  $m$.  In addition, $\|O(f)\|_1 \le \|f\|_1$.
\end{prop}
\begin{proof}  We first use the above calculations  for $\mathcal B_1$-measurable functions.  Given $f$ which is $\mathcal L_1$-measurable and $m$-integrable, there are $\mathcal B_1$-measurable functions $g_1,g_2$ with $g_1 \le g \le g_2$ and $g_1 = g_2 = g$ a.e with respect to $m$.  We would have to allow the Borel measurable functions here to take the values $\infty$ or $-\infty$ on $m$-null sets.  Now 
we see that $O(g_i)$ are both well-defined, up to allowing for infinity values on $m$-null sets, and $O(g_1) =O(g_2)$.  But then $f$ must be measurable and integrable with respect to $(\mu^*\times m)^*$.  Then we apply the Fubini-Tonelli Theorem for complete measures and the completion of the product measure as in Theorem 8.12 in Rudin~\cite{Rudin1}.  This gives all the results in the proposition.
\end{proof}

\begin{rem} Of course, the two definitions of $O$ above give the same operator.  The only difference is the approach with completion of measures allows us to give meaning to $\int f(x-y)\, d\mu(y)$ for a.e. $x [m]$, by replacing $\mu$ by $\mu^*$.  This of course does not change the value of this integral if $f$ is $\mathcal B_1$-measurable, at least it does not for a.e. $x [m]$.
\qed
\end{rem}

\bibliographystyle{amsplain}

\end{document}